\documentclass[12pt]{article}
\usepackage[a4paper]{geometry}

\usepackage[all]{xy}
\SelectTips{eu}{12}

\usepackage{amssymb, amsthm, amsmath, stmaryrd, eucal}

\usepackage{mathptmx}
\usepackage[scaled=.90]{helvet}
\usepackage{courier}

\usepackage[pdftex]{hyperref}

\hypersetup{
  colorlinks = true,
  linkcolor = black,
  urlcolor = blue,
  citecolor = black,
}

\newcommand{\bbF}{\mathbb{F}}
\newcommand{\bbP}{\mathbb{P}}
\newcommand{\bbS}{\mathbb{S}}
\newcommand{\bbZ}{\mathbb{Z}}

\newcommand{\calC}{\mathcal{C}}
\newcommand{\calE}{\mathcal{E}}
\newcommand{\calS}{\mathcal{S}}

\newcommand{\rmB}{\mathrm{B}}
\newcommand{\rmD}{\mathrm{D}}
\newcommand{\rmE}{\mathrm{E}}
\newcommand{\rmG}{\mathrm{G}}
\newcommand{\rmH}{\mathrm{H}}
\newcommand{\rmK}{\mathrm{K}}
\newcommand{\rmL}{\mathrm{L}}

\newcommand{\rmS}{\mathrm{S}}
\newcommand{\rmW}{\mathrm{W}}

\newcommand{\h}{\mathrm{h}}

\newcommand{\hatS}{\widehat{\bbS}}
\newcommand{\SSp}{\bbS/\!\!/p}
\newcommand{\SSt}{\bbS/\!\!/2}
\newcommand{\SSzeta}{\hatS/\!\!/\zeta}

\newcommand{\id}{\mathrm{id}}
\newcommand{\ch}{\mathrm{char}}

\newcommand{\ev}{\operatorname{ev}}

\newcommand{\BP}{\operatorname{BP}}
\newcommand{\BoP}{\operatorname{BoP}}
\newcommand{\BO}{\operatorname{BO}}
\newcommand{\EO}{\operatorname{EO}}
\newcommand{\MO}{\operatorname{MO}}
\newcommand{\MSO}{\operatorname{MSO}}
\newcommand{\MSpin}{\operatorname{MSpin}}
\newcommand{\MSU}{\operatorname{MSU}}
\newcommand{\SU}{\operatorname{SU}}
\newcommand{\MString}{\operatorname{MString}}
\newcommand{\MUsix}{\operatorname{MU}\langle6\rangle}

\newcommand{\tmf}{\operatorname{tmf}}
\newcommand{\ku}{\operatorname{ku}}
\newcommand{\ko}{\operatorname{ko}}
\newcommand{\KU}{\operatorname{KU}}
\newcommand{\KO}{\operatorname{KO}}

\newtheorem{proposition}{Proposition}[section]

\theoremstyle{definition}
\newtheorem{remark}[proposition]{Remark}
\newtheorem{definition}[proposition]{Definition}

\newtheorem{example}[proposition]{Example}

\numberwithin{equation}{section}

\begin{document}

\title{String bordism and chromatic characteristics}

\author{Markus Szymik}

\date{\today}

\maketitle

\renewcommand{\abstractname}{}

\begin{abstract}

\noindent 
We introduce characteristics into chromatic homotopy theory. This parallels the prime characteristics in number theory as well as in our earlier work on structured ring spectra and unoriented bordism theory. Here, the~$\rmK(n)$-local Hopkins-Miller classes~$\zeta_n$ take the places of the prime numbers, and this allows us to discuss higher bordism theories. We prove that the~$\rmK(2)$-localizations of the spectrum of topological modular forms as well as the string bordism spectrum have characteristic~$\zeta_2$.

\vspace{\baselineskip}

\end{abstract}


\section{Introduction}

The classification of manifolds is intimately tied to the homotopy theory of Thom spaces and spectra. If~$\MO$ denotes the Thom spectrum for the family of orthogonal groups, then its homotopy groups~$\pi_d\MO$ are given by the groups of bordism classes of~$d$-dimensional closed manifolds. Variants of this correspondence apply to manifolds with extra structure, such as orientations and Spin structures, for example. Arguably the most relevant of these variants for geometry are ordered into a hierarchy given by the higher connective covers~\hbox{$\BO\langle k\rangle\to\BO$} of~$\BO$, and their Thom spectra~$\MO\langle k\rangle$. For small values of~$k$, these describe the unoriented ($\MO\langle 1\rangle=\MO$), oriented ($\MO\langle 2\rangle=\MSO$), Spin~($\MO\langle 4\rangle=\MSpin$), and String bordism groups of manifolds ($\MO\langle 8\rangle=\MString$). The name `string' in this context appears to be due to Miller, see~\cite{Laures:q}. The spectra~$\MO\langle k\rangle$ are also interesting as approximations to the sphere spectrum~$\bbS$ itself, in a sense that can be made precise~\cite[Proposition~2.1.1]{Hovey:Duke}:~$\bbS\simeq\lim_k\MO\langle k\rangle$. The geometric relevance of the sphere spectrum stems, of course, from the fact that it is the Thom spectrum for stably framed manifolds.


All the bordism spectra that were just mentioned are canonically commutative ring spectra in the best sense of the word, namely~$\rmE_\infty$ ring spectra \cite{Eoo}. In fact, this concept was more or less invented in order to deal with the very examples of Thom spectra \cite{LMS}. The multiplicative structure allows us to study them through their {\it genera}: maps {\em out of} commutative bordism ring spectra into spectra which are easier to understand. This has been rather successful for small values of~$k$, and the following diagram summarizes the situation.
\begin{center}
  \mbox{ 
    \xymatrix@R=20pt{
    \vdots\ar[d] &\\
    \MString\ar[d]\ar[r] & \tmf\\
    \MSpin\ar[d]\ar[r] & \ko\\
    \MSO\ar[d]\ar[r] & \rmH\bbZ\\
    \MO\ar[r] & \rmH\bbF_2
}
}
\end{center}
Here, the spectra~$\rmH\bbF_2$ and~$\rmH\bbZ$ are the Eilenberg-Mac Lane spectra of the indicated rings, and the genera count the number of points mod~$2$ and with signs, respectively. In the next line, the spectrum~$\ko$ is the connective real K-theory spectrum that receives the topological~$\widehat{A}$-genus (or Atiyah-Bott-Shapiro orientation), compare~\cite{Joachim:1} and \cite{Joachim:2}. Finally, the spectrum~$\tmf$ is the spectrum of topological modular forms that was constructed in order to refine the Witten genus (or~$\sigma$-orientation), see \cite{Hopkins:ICM1}, \cite{Hopkins:ICM2}, \cite{Ando+Hopkins+Strickland:1}, \cite{Ando+Hopkins+Strickland:2}, and~\cite{Ando+Hopkins+Rezk}.


{\it Characteristics} in the sense of the title appear in the approach that is dual to the idea underlying genera. Namely, there are interesting ring spectra that come with maps {\em into} these bordism spectra. For example, since the unoriented bordism ring~$\pi_*\MO$ has characteristic~$2$, there is a unique (up to homotopy) map~$\SSt\to\MO$ of~$\rmE_\infty$ ring spectra from the versal~$\rmE_\infty$ ring spectrum~$\SSt$ of characteristic~$2$, see~\cite{Szymik}, where~$\rmE_\infty$ ring spectra of prime characteristics, and their versal examples~$\SSp$, have been studied from this point of view. However, the fact that~$\pi_0\MO\langle k\rangle=\bbZ$ as soon as~$k\geqslant 2$ makes it evident that ordinary prime characteristics have nothing to say about higher bordism theories. 


In order to gain a better understanding of higher bordism theories, we propose in this paper to replace the ordinary primes~$p\in\bbZ=\pi_0\bbS$ by something more elaborate, some classes that only appear after passing to the (Bousfield~\cite{Bousfield}) localization~$\hatS$ of the sphere spectrum~$\bbS$ with respect to any given Morava K-theory~$\rmK(n)$: the classes~$\zeta_n$ in~$\pi_{-1}\hatS$ which were first defined by Hopkins and Miller. See~\cite{Hovey:conjecture},~\cite{Devinatz+Hopkins}, and our exposition in Section~\ref{sec:zeta}. Just as~$\SSt$ has been used in~\cite{Szymik} to study the unoriented bordism spectrum~$\MO$, the aim of the present writing is to show that it is the spectra~$\SSzeta_n$ which are likewise relevant to the study of the chromatic localizations of higher bordism spectra. Incidentally, they also play an important role for chromatic homotopy theory in itself, but an exposition of those ideas will have to await another occasion.


Whenever~$A$ is any~$\rmK(n)$-local~$\rmE_\infty$ ring spectrum with unit~$u_A\colon\hatS\to A$, there is a naturally associated class
\[
\zeta_n(A)\colon \rmS^{-1}\overset{\zeta_n}{\longrightarrow}\rmS^0\overset{u_A}{\longrightarrow}A
\]
in~$\pi_{-1}A$. Continuing to use the terminology as in~\cite{Szymik}, we will say that a~$\rmK(n)$-local~$\rmE_\infty$ ring spectrum~$A$ has {\it (chromatic) characteristic~$\zeta_n$} if there exists a homotopy~\hbox{$\zeta_n(A)\simeq0$}, compare Definition~\ref{def:char} below. We note that this just defines a property of such ring spectra, and that for structural purposes one will want to work with actual choices of homotopies, i.e.~with commutative~$\SSzeta_n$-algebras. See Section~\ref{subsec:versal}, and~\cite{Szymik} again.

The value of these concepts can be measured by the wealth of relevant examples: The~$\rmK(1)$-localizations of~$\ko$,~$\ku$,~$\MSpin$, and~$\MSU$ all have characteristic~$\zeta_1$, see Propositions~\ref{prop:ko} and~\ref{prop:MSpin}. Some of this rephrases parts of Hopkins' unpublished preprint~\cite{Hopkins:Preprint}, and also the first versal example~$\SSzeta_1$ already appears in there, albeit in a different guise. Additional arguments, and the extensions to the named bordism spectra have been published by Laures, see~\cite{Laures:Spin} and~\cite{Laures:TMF}. 

In the genuinely new examples dealt with here, we take the natural next step: The~$\rmK(2)$-localizations of the topological modular forms spectrum~$\tmf$, the String bordism spectrum~$\MString$, and~$\MUsix$ all have characteristic~$\zeta_2$, see Propositions~\ref{prop:tmf} and~\ref{prop:Melliptic}.

There are other families of examples of characteristic~$\zeta_n$ spectra for arbitrary~$n$: the Lubin-Tate spectra~$\rmE_n$ (Example~\ref{ex:E}), the Iwasawa extensions~$\rmB_n$ of the~$\rmK(n)$-local sphere (Example~\ref{ex:B}), and the versal examples~$\SSzeta_n$ that map to all of these, see Proposition~\ref{prop:versal}.

We will use the following conventions: All spectra are implicitly~$\rmK(n)$-localized. In particular, the notation~$X\wedge Y$ will refer to the~$\rmK(n)$-localization of the usual smash product, and the homology~$X_0Y$ is defined as~$\pi_0$ of that. As an exemption to these rules, we will write~$\widehat\bbS$ for the~$\rmK(n)$-local sphere to emphasize the idea that it is a completed form of the sphere spectrum~$\bbS$, and~$\rmS^n=\Sigma^n\hatS$ for its (de)suspensions. 


\section{Characteristics in chromatic homotopy theory}\label{sec:zeta}

In this section we will review some chromatic homotopy theory as far as it is needed for our purposes, and introduce the basic concept of chromatic characteristics, see Section~\ref{subsec:char}. The case~$n=1$ will be mentioned as an accompanying example throughout, but we emphasize that this case is always somewhat atypical, and the general case is the one we are interested in. Also, in the spirit of \cite{Hovey+Strickland}, we have chosen notation that avoids having to say anything special when~$p=2$. Nevertheless, we do so, if it seems appropriate for the examples at hand, in particular in Section~\ref{sec:bordism} when it comes to bordism theories.

\subsection{The Lubin-Tate spectra}

Let~$p$ be a prime number, and~$n$ a positive integer. We will denote by~$\rmE_n$ the corresponding Lubin-Tate spectrum. The coefficient ring is isomorphic to
\[
\pi_*\rmE_n\cong \rmW(\bbF_{p^n})[[u_1,\dots,u_{n-1}]][u^{\pm1}],
\]
where~$\rmW$ is the Witt vector functor from commutative rings to commutative rings. This coefficient ring (or rather its formal spectrum) is a base for the universal formal deformation of the Honda formal group of height~$n$. 

\begin{example}
If~$n=1$, then the Lubin-Tate spectrum~$\rmE_1$ is the~$p$-adic completion~$\KU_p$ of the complex topological~$\rmK$-theory spectrum~$\KU$.
\end{example}

\subsection{The Morava groups}

The~$n$-th Morava stabilizer group~$\rmS_n$ and the Galois group of~$\bbF_{p^n}$ over~$\bbF_p$ both act on~$\rmE_n$ such that their semi-direct product~$\rmG_n$, the {\it extended} Morava stabilizer group, also acts on~$\rmE_n$.

\begin{example}
If~$n=1$, then the Morava stabilizer group~$\rmG_1=\bbS_1$ is the group~$\bbZ_p^\times$ of~$p$-adic units which acts on~$\rmE_1=\KU_p$ via Adams operations. 
\end{example}


\subsection{Devinatz-Hopkins fixed point spectra}

If~$K\leqslant\rmG_n$ is a closed subgroup of the extended Morava stabilizer group~$\rmG_n$, then~$\rmE_n^{\h K}$ will denote the corresponding Devinatz-Hopkins fixed point spectrum~\cite{Devinatz+Hopkins}. For example, in the maximal case~$K=\rmG_n$, Morava's change-of-rings theorem gives~$\rmE_n^{\h G}=\widehat\bbS$.

The Devinatz-Hopkins fixed point spectra are well under control in the optic of their Morava modules: There are isomorphisms
\[
(\rmE_n)_*(\rmE_n^{\h K})=\pi_*(\rmE_n\wedge \rmE_n^{\h K})\cong\calC(\rmG/K,\pi_*\rmE_n).
\]
For the trivial group~$K=e$ this has been known to Morava (and certainly others) for a long time. See~\cite{Hovey:Operations} for the history and a careful exposition.


\subsection{Some subgroups of the Morava stabilizer group}

The Morava stabilizer group acts on the Dieudonn\'e module of the Honda formal group of height~$n$, which is free of rank~$n$ over~$\rmW(\bbF_{p^n})$. The determinant gives a homomorphism~$\rmS_n\rightarrow\rmW(\bbF_{p^n})^\times$. This extends over~$\rmG_n$ and factors through~$\bbZ_p^\times$. The subgroup~$\rmS\rmG_n$ is defined as the kernel of the (surjective) determinant, so that we have an extension
\[
1\longrightarrow\rmS\rmG_n\longrightarrow\rmG_n\longrightarrow\bbZ_p^\times\longrightarrow1
\]
of groups. Let~$\Delta\leqslant\bbZ_p^\times$ denote the torsion subgroup. If~$p=2$, then this subgroup is cyclic of order~$2$, and if~$p\not=2$, then it is cyclic of order~$p-1$. The pre-image of~$\Delta$ under the determinant is customarily denoted by~$\rmG_n^1$. In other words, there is an extension
\begin{equation}\label{eq:residual}
1\longrightarrow\rmG_n^1\longrightarrow\rmG_n\longrightarrow\bbZ_p^\times/\Delta\longrightarrow1,
\end{equation}
and the groups~$\rmS\rmG_n$ and~$\rmG_n^1$ are then also related by a short exact sequence
\[
1\longrightarrow\rmS\rmG_n\longrightarrow\rmG_n^1\longrightarrow\Delta\longrightarrow1.
\]
We remark that there are (abstract) isomorphisms~$\bbZ_p^\times/\Delta\cong\bbZ_p$ of groups, but no canonical choice seems to be available.


\subsection{The Iwasawa extensions of the local spheres}\label{sec:B}

An Iwasawa extension is a (pro-)Galois extension (for example of a number field) with Galois group isomorphic to the additive group~$\bbZ_p$ of~$p$-adic integers for some prime number~$p$. The canonical Iwasawa extensions of the~$\rmK(n)$-local sphere is the Devinatz-Hopkins fixed point spectrum
\[
\rmB_n=\rmE_n^{\h\rmG^1_n}
\]
with respect to the closed subgroup~$\rmG^1_n$. This is sometimes referred to as {\it Mahowald's half-sphere}, in particular in the case~$n=2$. 

\begin{example}
If~$n=1$, then the spectrum~$\rmB_1$ is either the~$2$-completion~$\KO_2$ of the real topological~$\rmK$-theory spectrum~$\KO$~(when~\hbox{$p=2$}) or the Adams summand~$\rmL_p$ of the~$p$-completion of the complex topological~$\rmK$-theory spectrum~$\KU$~(when~\hbox{$p\not=2$}).
\end{example}

The spectra~$\rmB_n$ are well under control in the optic of their Morava modules: There are isomorphisms
\[
(\rmE_n)_*(\rmB_n)=\pi_*(\rmE_n\wedge\rmB_n)\cong\calC(\bbZ_p^\times/\Delta,\pi_*\rmE_n),
\]
and the right hand side can be identified (non-canonically) with the ring of continuous functions on the~$p$-adic affine line~$\bbZ_p$.

From~\eqref{eq:residual} we infer that the spectrum~$\rmB_n$ carries a residual action of~$\bbZ_p^\times/\Delta\cong\bbZ_p$, and this makes~$\hatS\to\rmB_n$ into an Iwasawa extension of the~$\rmK(n)$-local sphere. Whenever we chose a topological generator of this group, this yields 
an automorphism~$g\colon\rmB_n\rightarrow\rmB_n$.

\begin{proposition}{\bf(\cite[Proposition 8.1]{Devinatz+Hopkins})}
There is a homotopy fibration sequence
\begin{equation}\label{eq:fibration}
\rmS^0\overset{u_B}{\longrightarrow}
\rmB_n\overset{g-\id}{\longrightarrow}
\rmB_n\overset{\partial}{\longrightarrow}
\rmS^1
\end{equation}
of~$\rmK(n)$-local spectra.
\end{proposition}

For each~$p$ and~$n$, we fix one such fibration sequence once and for all.

\begin{example}
If~$n=1$, then the fibration sequence
\begin{equation}\label{eq:K(1)}
  \rmS^0\longrightarrow\rmB_1\overset{g-\id}{\longrightarrow}\rmB_1
\end{equation}
has been known for a long time. It can be extracted from~\cite{Bousfield}, which in turn relies on work of Mahowald ($p=2$) and Miller ($p\not=2$).
\end{example}


\subsection{The Hopkins-Miller classes}

We are now ready to introduce the Hopkins-Miller classes~$\zeta_n$ that are to play the role of the integral primes~$p$ in the chromatic context.

\begin{definition}
The homotopy class~$\zeta_n\in\pi_{-1}\widehat\bbS$ is defined as the (de-suspension of) the composition
\[
\rmS^0\overset{u_\rmB}{\longrightarrow}
\rmB_n\overset{\partial}{\longrightarrow}
\rmS^1
\]
of the outer maps in the homotopy fibration sequence~\eqref{eq:fibration}.
\end{definition}

If we map~$\rmS^0$ into the fibration sequence~\eqref{eq:fibration}, then we obtain a long exact sequence
\begin{equation}\label{eq:les}
[\rmS^0,\rmB_n]\overset{g_*-\id}{\longrightarrow}
[\rmS^0,\rmB_n]\overset{\partial_*}{\longrightarrow}
[\rmS^0,\rmS^1]\longrightarrow[\rmS^{-1},\rmB_n],
\end{equation}
and~$\zeta_n$ is, by definition, the image of the unit~$u_\rmB$ under the map~$\partial_*$. 

Using the defining homomorphism~\eqref{eq:residual} of~$\rmG_n^1$, we obtain a homomorphism
\[
c_n\colon\rmG_n\longrightarrow\bbZ_p^\times/\Delta\cong\bbZ_p\subseteq\pi_0\rmE_n
\]
that we can think of as a twisted homomorphism, or 1-cocyle, and as such it defines a class in the first continuous cohomology~$\rmH^1(\rmG_n;\pi_0\rmE_n)$.

\begin{proposition}\label{prop:zetanot0}{\bf(\cite[Proposition 8.2]{Devinatz+Hopkins})}
The Hopkins-Miller class~$\zeta_n$ is detected by~$\pm c_n$ in the~$\rmK(n)$-local~$\rmE_n$-based Adams-Novikov spectral sequence.
\end{proposition}

The class~$c_n$ is non-zero and~$\zeta_n$ is non-zero in~$\pi_{-1}\rmS^0$. In fact, it generates a subgroup isomorphic to~$\bbZ_p$ in~$\pi_{-1}\rmS^0$. However, the class~$\zeta_n$ becomes zero in~$\rmB_n$:

\begin{proposition}\label{prop:zeta_B=0}
The composition
\begin{center}
  \mbox{ 
    \xymatrix@1{\rmS^{-1}\ar[r]^{\zeta_n}&
    \rmS^0\ar[r]^{u_\rmB}&
  \rmB_n
} 
  }
\end{center}
of~$\zeta_n$ with the unit~$u_\rmB$ of~$\rmB_n$ is zero.
\end{proposition}

\begin{proof}
We have~$u_\rmB\zeta_n=u_\rmB\partial u_\rmB$, and there is already a homotopy~$u_\rmB\partial\simeq0$ as part of the fibration sequence~\eqref{eq:fibration}.
\end{proof}

\begin{remark}
If we knew that~$\pi_{-1}\rmB_n=0$ holds, which is a potentially stronger statement than the one in Proposition~\ref{prop:zeta_B=0}, and if we knew in addition that~$g_*=\id$ on~$\pi_0\rmB_n$, then the homomorphism~$\partial_*$ in~\eqref{eq:les} were an isomorphism between~$\pi_0\rmB_n$ and~$\pi_{-1}\widehat\bbS$. Both properties clearly hold in the case~$n=1$. As Hans-Werner Henn has assured, they also hold in the case~$n=2$ and~$p\geqslant3$: there are gaps around the groups~$\pi_{-3}\rmB_n$ and~$\pi_0\rmB_n$, which are both isomorphic to~$\bbZ_p$.
\end{remark}


\subsection{Chromatic characteristics}\label{subsec:char}

In the predecessor~\cite{Szymik} of this paper, we have defined the notion of an~$\rmE_\infty$ ring spectrum~$A$ of prime characteristic. If~$p$ is the prime number in question, then this means that there is a null-homotopy~\hbox{$p\simeq0$} in~$A$. We will now work~$\rmK(n)$-locally and replace the prime numbers~$p$ by the Hopkins-Miller classes~$\zeta_n$.

\begin{definition}\label{def:char}
If~$A$ is a~$\rmK(n)$-local~$\rmE_\infty$ ring spectrum with unit~$u_A\colon\hatS\to A$, then 
\[ 
\zeta_n(A)\colon
  \rmS^{-1}\overset{\zeta_n}{\longrightarrow}\widehat\bbS\overset{u_A}{\longrightarrow}A
\]
is the associated class in~$\pi_{-1}A$. 
If~$A$ is a~$\rmK(n)$-local~$\rmE_\infty$ ring spectrum such that there exists a null-homotopy~$\zeta_n(A)\simeq0$, then we will say that~$A$ has {\it characteristic~$\zeta_n$}, and we may also write
\[
\ch(A)=\zeta_n
\] 
in that case.
\end{definition}

We note that we have just defined a property of~$\rmK(n)$-local~$\rmE_\infty$ ring spectra. We have not, as the notation may suggest, defined a function~$\ch$ that is defined for each such ring spectrum. 

\begin{proposition}\label{prop:algebras}
If~$A$ is a~$\rmK(n)$-local~$\rmE_\infty$ ring spectrum of characteristic~$\zeta_n$, then so is every~$A$-algebra~$B$.
\end{proposition}

\begin{proof}
The unit of any~$A$-algebra~$B$ factors through the unit of~$A$.
\end{proof}

In order the demonstrate the relevance of the concept of (chromatic) characteristics outside of chromatic homotopy theory itself, we will now give many examples of naturally occurring~$\rmK(n)$-local~$\rmE_\infty$ ring spectra of characteristic~$\zeta_n$.


\section{String bordism and other examples}\label{sec:bordism}

In this section, we prove that (the~$\rmK(n)$-localizations of) several relevant examples of~$\rmE_\infty$ ring spectra have characteristic~$\zeta_n$. We start with some general results and then proceed to discuss the geometrically relevant cases~$n=1$ and~$n=2$. We also present non-examples, of course. The final Section~\ref{subsec:versal} introduces the versal examples~$\SSzeta_n$ that map to any given~$\rmK(n)$-local~$\rmE_\infty$ ring spectrum of characteristic~$\zeta_n$.


\subsection{General results and remarks}

First of all, here is an example which shows that not all~$\rmK(n)$-local~$\rmE_\infty$ ring spectra have characteristic~$\zeta_n$.

\begin{example}
In the initial example~$A=\widehat\bbS$ of the~$\rmK(n)$-local sphere, the unit is the identity, so that we have~\hbox{$\zeta_n(\widehat\bbS)=\zeta_n$}, 
and this is non-zero as a consequence of Proposition~\ref{prop:zetanot0}. Therefore,
\[
\ch(\widehat\bbS)\not=\zeta_n.
\]
This result is analogous to the fact that~$\ch(\bbS)\not=p$ for the (un-localized) ring of spheres.
\end{example}

Clearly, if~$A$ is a~$\rmK(n)$-local~$\rmE_\infty$ ring spectrum such that~$\pi_{-1}A$ vanishes, then the element~\hbox{$\zeta_n(A)\in\pi_{-1}A=0$} is automatically null-homotopic. Let us mention a couple of interesting examples of this type.

\begin{example}\label{ex:E}
Because the Lubin-Tate spectra~$\rmE_n$ are even, we have~$\pi_{-1}\rmE_n=0$, so that~$\zeta_n(\rmE_n)\simeq0$, and this implies
\[
\ch(\rmE_n)=\zeta_n.
\] 
In other words, the Lubin-Tate spectra~$\rmE_n$ all have characteristic~$\zeta_n$.
\end{example}

Even if we have~$\pi_{-1}A\not=0$, or if we are perhaps in a situation when we do not know yet whether this or~\hbox{$\pi_{-1}A=0$} holds, we might still be able to decide if~$\zeta_n(A)$ is null-homotopic. This is the case in the following examples. 

\begin{example}\label{ex:B}
We have
\[
\ch(\rmB_n)=\zeta_n
\]
for all~$n$ by Proposition~\ref{prop:zeta_B=0}.
\end{example}

After these examples that are available for every positive height~$n$, let us now turn towards the geometrically relevant cases, which require~$n=1$ and~$n=2$. 


\subsection{Multiplicative examples}

Let us first discuss examples with height~$n=1$.

\begin{proposition}\label{prop:ko}
We have
\[
\ch(\ko_{\rmK(1)})=\ch(\ku_{\rmK(1)})=\zeta_1
\]
at all primes.
\end{proposition}

\begin{proof}
The~$\rmK(1)$-localizations (at~$p$) agree with the~$p$-completions of the periodic versions, compare \cite[Lemma~2.3.5]{Hovey:Duke}. The complete periodic theories are well known to have vanishing~$\pi_{-1}$.
\end{proof}

\begin{proposition}\label{prop:MSpin}
We have
\[
\ch(\MSpin_{\rmK(1)})=\ch(\MSU_{\rmK(1)})=\zeta_1
\]
at all primes.
\end{proposition}

\begin{proof}
At odd primes both~$\MSpin$ and~$\MSU$ are equivalent to wedges of even suspensions  of~$\BP$, so that their~$\rmK(1)$-localizations are well understood. The~$\rmK(1)$-localization of~$\BP$ is 
\[
\BP_{\rmK(1)}=(v_1^{-1}\BP)_p.
\]
See~\cite[Lemma~2.3]{Hovey:conjecture}, for example. We see that~$\pi_*\BP_{\rmK(1)}$ is concentrated in even degrees. This implies~$\pi_{-1}=0$ for both of the spectra~$\MSpin$ and~$\MSU$, and {\it a fortiori}~$\ch=\zeta_1$ for both of them.

The case~$p=2$ needs some extra arguments. Since the Spin bordism spectrum~$\MSpin$, as any Thom spectrum, is connective, we in particular have~\hbox{$\pi_{-1}\MSpin=0$}. This does not imply the result for the~$\rmK(1)$-localization, however, think of~$\bbS$. But, the Anderson-Brown-Peterson splitting shows that this still holds after~$\rmK(1)$-localization, since the spectrum~$\MSpin$ splits~$\rmK(1)$-locally at the prime~$p=2$ as a wedge of localizations of copies of the spectrum~$\ko$, compare~\cite[Proposition~2.3.6]{Hovey:Duke}. Therefore, the result follows from what we have said for the~$\rmK$-theories. For~$\MSU$, replace the ABS-splitting by the results in~\cite{Botvinnik}, \cite{Kochman}, and~\cite{Pengelly}. 
\end{proof}

The higher multiplicative structure of the~$\rmK(1)$-localization of~$\MSU$ at the prime~$p=2$ has been investigated further in Reeker's thesis~\cite{Reeker}.


\subsection{Elliptic examples}

We will now turn towards examples with height~$n=2$.

\begin{proposition}\label{prop:tmf}
We have
\[
\ch(\tmf_{\rmK(2)})=\zeta_2
\]
at all primes.
\end{proposition}

\begin{proof}
In \cite[Remark~1.7.3]{Behrens}, Behrens has given an argument for the identification of the~$\rmK(2)$-localization of the spectrum of topological modular forms with~$\EO_2$, the homotopy fixed point spectrum of~$\rmE_2$ with respect to the maximal finite subgroup~$M$ of the extended Morava group~$\rmG_2$, that holds for the prime~$p=3$. His argument can be adapted to the case~$p=2$ as well. This allows us to control the class in~$\rmH^1(\rmG_2;(\rmE_2)_0\tmf)$ that detects the image of~$\zeta_2$ in the~$\rmK(n)$-local~$\rmE_n$-based Adams-Novikov spectral sequence. This is class is the image of~$\pm c_2$ under the homomorphism
\[
\rmH^1(\rmG_2;(\rmE_2)_0)
\longrightarrow
\rmH^1(\rmG_2;(\rmE_2)_0\tmf)
\]
induced by the unit~$\hatS\to\tmf_{\rmK(2)}$. The unit is given by the Devinatz-Hopkins fixed points of~$\rmE_2$ with respect to the inclusion of the maximal subgroup~$M$ into~$\rmG_2$. Therefore, the image of the class~$\pm c_2$ is the composition
\[
M\overset{\leqslant}{\longrightarrow}\rmG_2\overset{\pm c_2}{\longrightarrow}\pi_0\rmE_2.
\]
Since the coefficient ring~$\pi_0\rmE_2$ is torsion-free, and~$M$ is finite, this image is automatically zero, so that the class in~$\rmH^1(\rmG_2;(\rmE_2)_0\tmf)$ that detects the image of~$\zeta_2$ necessarily vanishes.

The situation at large primes~$p\geqslant5$ is similar, but less well represented in the published literature. See~\cite{Behrens:TMF}: The~$\rmK(2)$-localization of~$\tmf$ is the spectrum of global sections of the derived structure sheaf of the completion of the moduli stack of generalized elliptic curves in characteristic~$p$ at the complement of the ordinary locus. This sheaf can be constructed using the Goerss-Hopkins-Miller theory of Lubin-Tate spectra. The upshot is that the spectra of sections are given by homotopy fixed points of Lubin-Tate spectra with respect to finite groups. The difference is that this time their orders are co-prime to the characteristic. In any case, we see that the same argument as for~$p=2$ and~$p=3$ can be applied.
\end{proof}

The argument just given shows a bit more: Since the maximal subgroup~$M$ sits inside the subgroup~$\rmG_2^1$, we even know that the~$\rmK(2)$-localization of the topological modular forms spectrum is a~$\rmB_2$-algebra. By Proposition~\ref{prop:algebras} and Example~\ref{ex:B}, we know that~$\ch(T)=\zeta_n$ for all~$\rmB_n$-algebras~$T$.

\begin{proposition}\label{prop:Melliptic}
We have
\[
\ch(\MString_{\rmK(2)})=\ch(\MUsix_{\rmK(2)})=\zeta_2
\]
at all primes.
\end{proposition}

\begin{proof}
This is again easier for large primes, and we will deal with these cases first. If~$p\geqslant5$, then both~$\MString$ and~$\MUsix$ split~$p$-locally as a wedge of even suspensions of~$\BP$ by \cite[Corollary~2.2]{Hovey+Ravenel}. See also \cite{Hovey:large_primes}. {\it A fortiori}, this holds also~$\rmK(2)$-locally at the prime in question. But the~$\rmK(2)$-localization of~$\BP$ is \begin{equation}\label{eq:BPloc}
\BP_{\rmK(2)}=(v_2^{-1}\BP)_{p,v_1}.
\end{equation}
See~\cite[Lemma~2.3]{Hovey:conjecture} again. We see that the homotopy groups~$\pi_*\BP_{\rmK(2)}$ are concentrated in even degrees. Therefore, we can deduce that
\[
\pi_{-1}\MString_{\rmK(2)}=\pi_{-1}\MUsix_{\rmK(2)}=0
\]
for large primes, from which the statement follows.

For the small primes~$p=2$ and~$p=3$, we have to work a bit harder. In these cases, we can still find finite complexes~$F$ with cells (depending on~$p$) only in even dimensions such that~$\MString\wedge F$ and~$\MUsix\wedge F$ split as wedges of (even) suspensions of~$\BP$, see~\cite[Corollary~2.2]{Hovey+Ravenel}. Strictly speaking, this excludes the case~$\MString$ at the prime~$p=2$. But, since there is a map~$\MUsix\to\MString$, Proposition~\ref{prop:algebras} guarantees that it is sufficient to prove our result for~$\MUsix$ to be able to infer it for~$\MString$ as well.

We will control the class in~$\rmH^1(\rmG_2;(\rmE_2)_0\MUsix)$ that detects~$\zeta_2$ in the~$\rmK(n)$-local~$\rmE_n$-based Adams-Novikov spectral sequence. Since~$F$ has cells only in even dimensions, the Atiyah-Hirzebruch sequence for~$(\rmE_2)_*F$ collapses, and it shows that this is a free~$(\rmE_2)_*$-module~(of rank the number of cells of~$F$). Then~$(\rmE_2)_*\MUsix$ is embedded into
\[
(\rmE_2)_*(\MUsix\wedge F)\cong(\rmE_2)_*\MUsix\otimes_{(\rmE_2)_*}(\rmE_2)_*F
\]
as a direct summand by the map~$\MUsix\to\MUsix\wedge F$ given by the bottom cell of~$F$. It follows that this map induces an injection
\[
\rmH^1(\rmG_2;(\rmE_2)_0\MUsix)\longrightarrow\rmH^1(\rmG_2;(\rmE_2)_0(\MUsix\wedge F)),
\]
so that it is more than enough to show that the target vanishes. This follows from the splitting of~$\MUsix\wedge F$: We know that the spectrum~$\BP_{\rmK(2)}$ is Landweber exact, for example by \eqref{eq:BPloc}. This implies that~$\BP_*(\BP_{\rmK(2)})$ is an extended~$\BP_*$-comodule, so that the~$\BP$-based Adams-Novikov spectral sequence collapses: the only non-zero terms on the~$\rmE_2$ page live on the~$0$-line. By the change-of-rings theorem, we infer that the groups on the~$\rmE_2$ page are given by~$\rmH^s(\rmG_2;(\rmE_2)_t\BP)$. We deduce that these groups vanish for positive~$s$. The same holds, of course, when we replace~$\BP$ with any of its suspensions.
\end{proof}


\subsection{Versal examples}\label{subsec:versal}

An important theoretical role in the theory of~$\rmK(n)$-local~$\rmE_\infty$ ring spectra of characteristic~$\zeta_n$ is played by the versal examples. These will be introduced now.

Let~$\bbP X$ denote the free~$\rmK(n)$-local~$\rmE_\infty$ ring spectrum on  a~$\rmK(n)$-local spectrum~$X$. There is an adjunction
\[
\calE^{\rmK(n)}_\infty(\bbP X,A)\cong\calS^{\rmK(n)}_\infty(X,A)
\]
between the space of~$\rmK(n)$-local~$\rmE_\infty$ ring maps and the space of maps of~$\rmK(n)$-local spectra, which sends an~$\rmE_\infty$ map~$\bbP X\rightarrow A$ to its restriction along the unit~$X\rightarrow\bbP X$ of the adjunction. (The unit of the~$\rmE_\infty$ ring spectrum~$\bbP X$ is a map~$\hatS\to\bbP X$, of course.) The inverse is denoted by~$x\mapsto\ev(x)$ for any given class~\hbox{$x\colon X\rightarrow A$}.

\begin{definition}
The~$\rmK(n)$-local~$\rmE_\infty$ ring spectrum~$\SSzeta_n$ is defined as a homotopy pushout
\begin{center}
  \mbox{ 
    \xymatrix{
      \bbP \rmS^{-1}\ar[r]^{\ev(0)}\ar[d]_{\ev(\zeta_n)}&\hatS\ar[d]^{} \\
      \hatS\ar[r]_{}&\SSzeta_n
    } 
  }
\end{center}
in the category of~$\rmK(n)$-local~$\rmE_\infty$ ring spectra. 
\end{definition}

There are various way of producing a homotopy pushout. The easiest one might be to start with a cofibrant model of~$\hatS$, replacing the morphism~$\ev(0)=\bbP(\rmS^{-1}\to\rmD^0)$ by~$\bbP$ of a cofibration~$\rmS^{-1}\to K$ for some contractible~$K$, for example the cone on~$\rmS^{-1}$, and then taking the actual pushout.

\begin{proposition}\label{prop:versal}
We have
\[
\ch(\SSzeta_n)=\zeta_n
\]
for all primes~$p$.
\end{proposition}

\begin{proof}
The homotopy commutativity of the enlarged diagram
\begin{center}
  \mbox{ 
    \xymatrix{
      \rmS^{-1}\ar@(r,u)[drr]^0\ar[dr]\ar@(d,l)[ddr]_{\zeta_n}&&\\
      &\bbP \rmS^{-1}\ar[r]\ar[d]&\hatS\ar[d]^{} \\
      & \hatS\ar[r]_{}&\SSzeta_n
    } 
  }
\end{center}
immediately shows that~$\zeta_n(\SSzeta_n)$ is homotopic to zero.
\end{proof}

\begin{remark}\label{rem:uniqueness}
The~$\rmK(n)$-local~$\rmE_\infty$ ring spectrum~$\SSzeta_n$ has the usual property of any homotopy pushout: a null-homotopy of~$\zeta_n(A)$ gives rise to a map~$\SSzeta_n\to A$, and conversely. In fact, this allows us to add upon the preceding proposition: Any choice of homotopy pushout~$\SSzeta_n$ comes with a {\em preferred} homotopy~$\zeta_n(\SSzeta_n)\simeq0$ (that corresponds to the identity map). It also implies that there is a map
\begin{equation}\label{eq:not_unique}
\SSzeta_n\longrightarrow A
\end{equation}
of~$\rmK(n)$-local~$\rmE_\infty$ ring spectra if and only if~$\ch(A)=\zeta_n$. There is no reason why a map~\eqref{eq:not_unique}, once it exists, should be unique. In fact, there will usually be many such maps, even up to homotopy. This explains our use of Artin's term `versal' rather than~`universal.'
\end{remark}


\section*{Acknowledgments}

This research has been supported by the Danish National Research Foundation through the Centre 
for Symmetry and Deformation (DNRF92). I would like to thank Hans-Werner Henn for a helpful conversation.



\vfill

\parbox{\linewidth}{
Department of Mathematical Sciences\\
University of Copenhagen\\
Universitetsparken 5\\
2100 Copenhagen \O\\
DENMARK\\
\phantom{ }\\
\href{mailto:szymik@math.ku.dk}{szymik@math.ku.dk}\\
\href{http://www.math.ku.dk/~xvd217}{www.math.ku.dk/$\sim$xvd217}}

\end{document}